\documentclass[reqno,11pt]{amsart}
\usepackage[foot]{amsaddr}

\usepackage{tikz}
\usepackage{tikz-cd}
\usepackage{amsmath}
\usepackage{amssymb}
\usepackage{amstext}
\usepackage{amsfonts}

\usepackage{xspace}
\usepackage{url}

\long\def\ig#1{\relax}
\ig{Thanks to Roberto Minio for this def'n.  Compare the def'n of
\comment in AMSTeX.}
 
\newcount \coefa
\newcount \coefb
\newcount \coefc
\newcount\tempcounta
\newcount\tempcountb
\newcount\tempcountc
\newcount\tempcountd
\newcount\xext
\newcount\yext
\newcount\xoff
\newcount\yoff
\newcount\gap%
\newcount\arrowtypea
\newcount\arrowtypeb
\newcount\arrowtypec
\newcount\arrowtyped
\newcount\arrowtypee
\newcount\height
\newcount\width
\newcount\xpos
\newcount\ypos
\newcount\run
\newcount\rise
\newcount\arrowlength
\newcount\halflength
\newcount\arrowtype
\newdimen\tempdimen
\newdimen\xlen
\newdimen\ylen
\newsavebox{\tempboxa}%
\newsavebox{\tempboxb}%
\newsavebox{\tempboxc}%
 
\makeatletter
\def\settypes(#1,#2,#3){\arrowtypea#1 \arrowtypeb#2 \arrowtypec#3}
\def\settoheight#1#2{\setbox\@tempboxa\hbox{#2}#1\ht\@tempboxa\relax}%
\def\settodepth#1#2{\setbox\@tempboxa\hbox{#2}#1\dp\@tempboxa\relax}%
\def\settokens[#1`#2`#3`#4]{%
     \def\tokena{#1}\def\tokenb{#2}\def\tokenc{#3}\def\tokend{#4}}
\def\setsqparms[#1`#2`#3`#4;#5`#6]{%
\arrowtypea #1
\arrowtypeb #2
\arrowtypec #3
\arrowtyped #4
\width #5
\height #6
}
\def\setpos(#1,#2){\xpos=#1 \ypos#2}
 
\def\bfig{\begin{picture}(\xext,\yext)(\xoff,\yoff)}
\def\efig{\end{picture}}
 
\def\putbox(#1,#2)#3{\put(#1,#2){\makebox(0,0){$#3$}}}
 
\def\settriparms[#1`#2`#3;#4]{\settripairparms[#1`#2`#3`1`1;#4]}%
 
\def\settripairparms[#1`#2`#3`#4`#5;#6]{%
\arrowtypea #1
\arrowtypeb #2
\arrowtypec #3
\arrowtyped #4
\arrowtypee #5
\width #6
\height #6
}
 
\def\resetparms{\settripairparms[1`1`1`1`1;500]\width 500}
 
\resetparms
 
\def\mvector(#1,#2)#3{
\put(0,0){\vector(#1,#2){#3}}%
\put(0,0){\vector(#1,#2){30}}%
}
\def\evector(#1,#2)#3{{
\arrowlength #3
\put(0,0){\vector(#1,#2){\arrowlength}}%
\advance \arrowlength by-30
\put(0,0){\vector(#1,#2){\arrowlength}}%
}}
 
\def\horsize#1#2{%
\settowidth{\tempdimen}{$#2$}%
#1=\tempdimen
\divide #1 by\unitlength
}
 
\def\vertsize#1#2{%
\settoheight{\tempdimen}{$#2$}%
#1=\tempdimen
\settodepth{\tempdimen}{$#2$}%
\advance #1 by\tempdimen
\divide #1 by\unitlength
}
 
\def\vertadjust[#1`#2`#3]{%
\vertsize{\tempcounta}{#1}%
\vertsize{\tempcountb}{#2}%
\ifnum \tempcounta<\tempcountb \tempcounta=\tempcountb \fi
\divide\tempcounta by2
\vertsize{\tempcountb}{#3}%
\ifnum \tempcountb>0 \advance \tempcountb by20 \fi
\ifnum \tempcounta<\tempcountb \tempcounta=\tempcountb \fi
}
 
\def\horadjust[#1`#2`#3]{%
\horsize{\tempcounta}{#1}%
\horsize{\tempcountb}{#2}%
\ifnum \tempcounta<\tempcountb \tempcounta=\tempcountb \fi
\divide\tempcounta by20
\horsize{\tempcountb}{#3}%
\ifnum \tempcountb>0 \advance \tempcountb by60 \fi
\ifnum \tempcounta<\tempcountb \tempcounta=\tempcountb \fi
}
 
\ig{ In this procedure, #1 is the paramater that sticks out all the way,
#2 sticks out the least and #3 is a label sticking out half way.  #4 is
the amount of the offset.}
 
\def\sladjust[#1`#2`#3]#4{%
\tempcountc=#4
\horsize{\tempcounta}{#1}%
\divide \tempcounta by2
\horsize{\tempcountb}{#2}%
\divide \tempcountb by2
\advance \tempcountb by-\tempcountc
\ifnum \tempcounta<\tempcountb \tempcounta=\tempcountb\fi
\divide \tempcountc by2
\horsize{\tempcountb}{#3}%
\advance \tempcountb by-\tempcountc
\ifnum \tempcountb>0 \advance \tempcountb by80\fi
\ifnum \tempcounta<\tempcountb \tempcounta=\tempcountb\fi
\advance\tempcounta by20
}
 
\def\putvector(#1,#2)(#3,#4)#5#6{{%
\xpos=#1
\ypos=#2
\run=#3
\rise=#4
\arrowlength=#5
\arrowtype=#6
\ifnum \arrowtype<0
    \ifnum \run=0
        \advance \ypos by-\arrowlength
    \else
        \tempcounta \arrowlength
        \multiply \tempcounta by\rise
        \divide \tempcounta by\run
        \ifnum\run>0
            \advance \xpos by\arrowlength
            \advance \ypos by\tempcounta
        \else
            \advance \xpos by-\arrowlength
            \advance \ypos by-\tempcounta
        \fi
    \fi
    \multiply \arrowtype by-1
    \multiply \rise by-1
    \multiply \run by-1
\fi
\ifnum \arrowtype=1
    \put(\xpos,\ypos){\vector(\run,\rise){\arrowlength}}%
\else\ifnum \arrowtype=2
    \put(\xpos,\ypos){\mvector(\run,\rise)\arrowlength}%
\else\ifnum\arrowtype=3
    \put(\xpos,\ypos){\evector(\run,\rise){\arrowlength}}%
\fi\fi\fi
}}
 
\def\putsplitvector(#1,#2)#3#4{
\xpos #1
\ypos #2
\arrowtype #4
\halflength #3
\arrowlength #3
\gap 140
\advance \halflength by-\gap
\divide \halflength by2
\ifnum \arrowtype=1
    \put(\xpos,\ypos){\line(0,-1){\halflength}}%
    \advance\ypos by-\halflength
    \advance\ypos by-\gap
    \put(\xpos,\ypos){\vector(0,-1){\halflength}}%
\else\ifnum \arrowtype=2
    \put(\xpos,\ypos){\line(0,-1)\halflength}%
    \put(\xpos,\ypos){\vector(0,-1)3}%
    \advance\ypos by-\halflength
    \advance\ypos by-\gap
    \put(\xpos,\ypos){\vector(0,-1){\halflength}}%
\else\ifnum\arrowtype=3
    \put(\xpos,\ypos){\line(0,-1)\halflength}%
    \advance\ypos by-\halflength
    \advance\ypos by-\gap
    \put(\xpos,\ypos){\evector(0,-1){\halflength}}%
\else\ifnum \arrowtype=-1
    \advance \ypos by-\arrowlength
    \put(\xpos,\ypos){\line(0,1){\halflength}}%
    \advance\ypos by\halflength
    \advance\ypos by\gap
    \put(\xpos,\ypos){\vector(0,1){\halflength}}%
\else\ifnum \arrowtype=-2
    \advance \ypos by-\arrowlength
    \put(\xpos,\ypos){\line(0,1)\halflength}%
    \put(\xpos,\ypos){\vector(0,1)3}%
    \advance\ypos by\halflength
    \advance\ypos by\gap
    \put(\xpos,\ypos){\vector(0,1){\halflength}}%
\else\ifnum\arrowtype=-3
    \advance \ypos by-\arrowlength
    \put(\xpos,\ypos){\line(0,1)\halflength}%
    \advance\ypos by\halflength
    \advance\ypos by\gap
    \put(\xpos,\ypos){\evector(0,1){\halflength}}%
\fi\fi\fi\fi\fi\fi
}
 
\def\putmorphism(#1)(#2,#3)[#4`#5`#6]#7#8#9{{%
\run #2
\rise #3
\ifnum\rise=0
  \puthmorphism(#1)[#4`#5`#6]{#7}{#8}{#9}%
\else\ifnum\run=0
  \putvmorphism(#1)[#4`#5`#6]{#7}{#8}{#9}%
\else
\setpos(#1)%
\arrowlength #7
\arrowtype #8
\ifnum\run=0
\else\ifnum\rise=0
\else
\ifnum\run>0
    \coefa=1
\else
   \coefa=-1
\fi
\ifnum\arrowtype>0
   \coefb=0
   \coefc=-1
\else
   \coefb=\coefa
   \coefc=1
   \arrowtype=-\arrowtype
\fi
\width=2
\multiply \width by\run
\divide \width by\rise
\ifnum \width<0  \width=-\width\fi
\advance\width by60
\if l#9 \width=-\width\fi
\putbox(\xpos,\ypos){#4}
{\multiply \coefa by\arrowlength
\advance\xpos by\coefa
\multiply \coefa by\rise
\divide \coefa by\run
\advance \ypos by\coefa
\putbox(\xpos,\ypos){#5} }%
{\multiply \coefa by\arrowlength
\divide \coefa by2
\advance \xpos by\coefa
\advance \xpos by\width
\multiply \coefa by\rise
\divide \coefa by\run
\advance \ypos by\coefa
\if l#9%
   \put(\xpos,\ypos){\makebox(0,0)[r]{$#6$}}%
\else\if r#9%
   \put(\xpos,\ypos){\makebox(0,0)[l]{$#6$}}%
\fi\fi }%
{\multiply \rise by-\coefc
\multiply \run by-\coefc
\multiply \coefb by\arrowlength
\advance \xpos by\coefb
\multiply \coefb by\rise
\divide \coefb by\run
\advance \ypos by\coefb
\multiply \coefc by70
\advance \ypos by\coefc
\multiply \coefc by\run
\divide \coefc by\rise
\advance \xpos by\coefc
\multiply \coefa by140
\multiply \coefa by\run
\divide \coefa by\rise
\advance \arrowlength by\coefa
\ifnum \arrowtype=1
   \put(\xpos,\ypos){\vector(\run,\rise){\arrowlength}}%
\else\ifnum\arrowtype=2
   \put(\xpos,\ypos){\mvector(\run,\rise){\arrowlength}}%
\else\ifnum\arrowtype=3
   \put(\xpos,\ypos){\evector(\run,\rise){\arrowlength}}%
\fi\fi\fi}\fi\fi\fi\fi}}
 
\def\puthmorphism(#1,#2)[#3`#4`#5]#6#7#8{{%
\xpos #1
\ypos #2
\width #6
\arrowlength #6
\putbox(\xpos,\ypos){#3\vphantom{#4}}%
{\advance \xpos by\arrowlength
\putbox(\xpos,\ypos){\vphantom{#3}#4}}%
\horsize{\tempcounta}{#3}%
\horsize{\tempcountb}{#4}%
\divide \tempcounta by2
\divide \tempcountb by2
\advance \tempcounta by30
\advance \tempcountb by30
\advance \xpos by\tempcounta
\advance \arrowlength by-\tempcounta
\advance \arrowlength by-\tempcountb
\putvector(\xpos,\ypos)(1,0){\arrowlength}{#7}%
\divide \arrowlength by2
\advance \xpos by\arrowlength
\vertsize{\tempcounta}{#5}%
\divide\tempcounta by2
\advance \tempcounta by20
\if a#8 %
   \advance \ypos by\tempcounta
   \putbox(\xpos,\ypos){#5}%
\else
   \advance \ypos by-\tempcounta
   \putbox(\xpos,\ypos){#5}%
\fi}}
 
\def\putvmorphism(#1,#2)[#3`#4`#5]#6#7#8{{%
\xpos #1
\ypos #2
\arrowlength #6
\arrowtype #7
\settowidth{\xlen}{$#5$}%
\putbox(\xpos,\ypos){#3}%
{\advance \ypos by-\arrowlength
\putbox(\xpos,\ypos){#4}}%
{\advance\arrowlength by-140
\advance \ypos by-70
\ifdim\xlen>0pt
   \if m#8%
      \putsplitvector(\xpos,\ypos){\arrowlength}{\arrowtype}%
   \else
      \putvector(\xpos,\ypos)(0,-1){\arrowlength}{\arrowtype}%
   \fi
\else
   \putvector(\xpos,\ypos)(0,-1){\arrowlength}{\arrowtype}%
\fi}%
\ifdim\xlen>0pt
   \divide \arrowlength by2
   \advance\ypos by-\arrowlength
   \if l#8%
      \advance \xpos by-40
      \put(\xpos,\ypos){\makebox(0,0)[r]{$#5$}}%
   \else\if r#8%
      \advance \xpos by40
      \put(\xpos,\ypos){\makebox(0,0)[l]{$#5$}}%
   \else
      \putbox(\xpos,\ypos){#5}%
   \fi\fi
\fi
}}
 
\def\topadjust[#1`#2`#3]{%
\yoff=10
\vertadjust[#1`#2`{#3}]%
\advance \yext by\tempcounta
\advance \yext by 10
}
\def\botadjust[#1`#2`#3]{%
\vertadjust[#1`#2`{#3}]%
\advance \yext by\tempcounta
\advance \yoff by-\tempcounta
}
\def\leftadjust[#1`#2`#3]{%
\xoff=0
\horadjust[#1`#2`{#3}]%
\advance \xext by\tempcounta
\advance \xoff by-\tempcounta
}
\def\rightadjust[#1`#2`#3]{%
\horadjust[#1`#2`{#3}]%
\advance \xext by\tempcounta
}
\def\rightsladjust[#1`#2`#3]{%
\sladjust[#1`#2`{#3}]{\width}%
\advance \xext by\tempcounta
}
\def\leftsladjust[#1`#2`#3]{%
\xoff=0
\sladjust[#1`#2`{#3}]{\width}%
\advance \xext by\tempcounta
\advance \xoff by-\tempcounta
}
\def\adjust[#1`#2;#3`#4;#5`#6;#7`#8]{%
\topadjust[#1``{#2}]
\leftadjust[#3``{#4}]
\rightadjust[#5``{#6}]
\botadjust[#7``{#8}]}
 
\def\putsquarep<#1>(#2)[#3;#4`#5`#6`#7]{{%
\setsqparms[#1]%
\setpos(#2)%
\settokens[#3]%
\puthmorphism(\xpos,\ypos)[\tokenc`\tokend`{#7}]{\width}{\arrowtyped}b%
\advance\ypos by \height
\puthmorphism(\xpos,\ypos)[\tokena`\tokenb`{#4}]{\width}{\arrowtypea}a%
\putvmorphism(\xpos,\ypos)[``{#5}]{\height}{\arrowtypeb}l%
\advance\xpos by \width
\putvmorphism(\xpos,\ypos)[``{#6}]{\height}{\arrowtypec}r%
}}
 
\def\putsquare{\@ifnextchar <{\putsquarep}{\putsquarep%
   <\arrowtypea`\arrowtypeb`\arrowtypec`\arrowtyped;\width`\height>}}
\def\square{\@ifnextchar< {\squarep}{\squarep
   <\arrowtypea`\arrowtypeb`\arrowtypec`\arrowtyped;\width`\height>}}
\def\squarep<#1>[#2`#3`#4`#5;#6`#7`#8`#9]{{
\setsqparms[#1]
\xext=\width                                          
\yext=\height                                         
\topadjust[#2`#3`{#6}]
\botadjust[#4`#5`{#9}]
\leftadjust[#2`#4`{#7}]
\rightadjust[#3`#5`{#8}]
\begin{picture}(\xext,\yext)(\xoff,\yoff)
\putsquarep<\arrowtypea`\arrowtypeb`\arrowtypec`\arrowtyped;\width`\height>%
(0,0)[#2`#3`#4`#5;#6`#7`#8`{#9}]%
\end{picture}%
}}
 
\def\putptrianglep<#1>(#2,#3)[#4`#5`#6;#7`#8`#9]{{%
\settriparms[#1]%
\xpos=#2 \ypos=#3
\advance\ypos by \height
\puthmorphism(\xpos,\ypos)[#4`#5`{#7}]{\height}{\arrowtypea}a%
\putvmorphism(\xpos,\ypos)[`#6`{#8}]{\height}{\arrowtypeb}l%
\advance\xpos by\height
\putmorphism(\xpos,\ypos)(-1,-1)[``{#9}]{\height}{\arrowtypec}r%
}}
 
\def\putptriangle{\@ifnextchar <{\putptrianglep}{\putptrianglep
   <\arrowtypea`\arrowtypeb`\arrowtypec;\height>}}
\def\ptriangle{\@ifnextchar <{\ptrianglep}{\ptrianglep
   <\arrowtypea`\arrowtypeb`\arrowtypec;\height>}}
 
\def\ptrianglep<#1>[#2`#3`#4;#5`#6`#7]{{
\settriparms[#1]%
\width=\height                         
\xext=\width                           
\yext=\width                           
\topadjust[#2`#3`{#5}]
\botadjust[#3``]
\leftadjust[#2`#4`{#6}]
\rightsladjust[#3`#4`{#7}]
\begin{picture}(\xext,\yext)(\xoff,\yoff)
\putptrianglep<\arrowtypea`\arrowtypeb`\arrowtypec;\height>%
(0,0)[#2`#3`#4;#5`#6`{#7}]%
\end{picture}%
}}
 
\def\putqtrianglep<#1>(#2,#3)[#4`#5`#6;#7`#8`#9]{{%
\settriparms[#1]%
\xpos=#2 \ypos=#3
\advance\ypos by\height
\puthmorphism(\xpos,\ypos)[#4`#5`{#7}]{\height}{\arrowtypea}a%
\putmorphism(\xpos,\ypos)(1,-1)[``{#8}]{\height}{\arrowtypeb}l%
\advance\xpos by\height
\putvmorphism(\xpos,\ypos)[`#6`{#9}]{\height}{\arrowtypec}r%
}}
 
\def\putqtriangle{\@ifnextchar <{\putqtrianglep}{\putqtrianglep
   <\arrowtypea`\arrowtypeb`\arrowtypec;\height>}}
\def\qtriangle{\@ifnextchar <{\qtrianglep}{\qtrianglep
   <\arrowtypea`\arrowtypeb`\arrowtypec;\height>}}
 
\def\qtrianglep<#1>[#2`#3`#4;#5`#6`#7]{{
\settriparms[#1]
\width=\height                         
\xext=\width                           
\yext=\height                          
\topadjust[#2`#3`{#5}]
\botadjust[#4``]
\leftsladjust[#2`#4`{#6}]
\rightadjust[#3`#4`{#7}]
\begin{picture}(\xext,\yext)(\xoff,\yoff)
\putqtrianglep<\arrowtypea`\arrowtypeb`\arrowtypec;\height>%
(0,0)[#2`#3`#4;#5`#6`{#7}]%
\end{picture}%
}}
 
\def\putdtrianglep<#1>(#2,#3)[#4`#5`#6;#7`#8`#9]{{%
\settriparms[#1]%
\xpos=#2 \ypos=#3
\puthmorphism(\xpos,\ypos)[#5`#6`{#9}]{\height}{\arrowtypec}b%
\advance\xpos by \height \advance\ypos by\height
\putmorphism(\xpos,\ypos)(-1,-1)[``{#7}]{\height}{\arrowtypea}l%
\putvmorphism(\xpos,\ypos)[#4``{#8}]{\height}{\arrowtypeb}r%
}}
 
\def\putdtriangle{\@ifnextchar <{\putdtrianglep}{\putdtrianglep
   <\arrowtypea`\arrowtypeb`\arrowtypec;\height>}}
\def\dtriangle{\@ifnextchar <{\dtrianglep}{\dtrianglep
   <\arrowtypea`\arrowtypeb`\arrowtypec;\height>}}
 
\def\dtrianglep<#1>[#2`#3`#4;#5`#6`#7]{{
\settriparms[#1]
\width=\height                         
\xext=\width                           
\yext=\height                          
\topadjust[#2``]
\botadjust[#3`#4`{#7}]
\leftsladjust[#3`#2`{#5}]
\rightadjust[#2`#4`{#6}]
\begin{picture}(\xext,\yext)(\xoff,\yoff)
\putdtrianglep<\arrowtypea`\arrowtypeb`\arrowtypec;\height>%
(0,0)[#2`#3`#4;#5`#6`{#7}]%
\end{picture}%
}}
 
\def\putbtrianglep<#1>(#2,#3)[#4`#5`#6;#7`#8`#9]{{%
\settriparms[#1]%
\xpos=#2 \ypos=#3
\puthmorphism(\xpos,\ypos)[#5`#6`{#9}]{\height}{\arrowtypec}b%
\advance\ypos by\height
\putmorphism(\xpos,\ypos)(1,-1)[``{#8}]{\height}{\arrowtypeb}r%
\putvmorphism(\xpos,\ypos)[#4``{#7}]{\height}{\arrowtypea}l%
}}
 
\def\putbtriangle{\@ifnextchar <{\putbtrianglep}{\putbtrianglep
   <\arrowtypea`\arrowtypeb`\arrowtypec;\height>}}
\def\btriangle{\@ifnextchar <{\btrianglep}{\btrianglep
   <\arrowtypea`\arrowtypeb`\arrowtypec;\height>}}
 
\def\btrianglep<#1>[#2`#3`#4;#5`#6`#7]{{
\settriparms[#1]
\width=\height                         
\xext=\width                           
\yext=\height                          
\topadjust[#2``]
\botadjust[#3`#4`{#7}]
\leftadjust[#2`#3`{#5}]
\rightsladjust[#4`#2`{#6}]
\begin{picture}(\xext,\yext)(\xoff,\yoff)
\putbtrianglep<\arrowtypea`\arrowtypeb`\arrowtypec;\height>%
(0,0)[#2`#3`#4;#5`#6`{#7}]%
\end{picture}%
}}
 
\def\putAtrianglep<#1>(#2,#3)[#4`#5`#6;#7`#8`#9]{{%
\settriparms[#1]%
\xpos=#2 \ypos=#3
{\multiply \height by2
\puthmorphism(\xpos,\ypos)[#5`#6`{#9}]{\height}{\arrowtypec}b}%
\advance\xpos by\height \advance\ypos by\height
\putmorphism(\xpos,\ypos)(-1,-1)[#4``{#7}]{\height}{\arrowtypea}l%
\putmorphism(\xpos,\ypos)(1,-1)[``{#8}]{\height}{\arrowtypeb}r%
}}
 
\def\putAtriangle{\@ifnextchar <{\putAtrianglep}{\putAtrianglep
   <\arrowtypea`\arrowtypeb`\arrowtypec;\height>}}
\def\Atriangle{\@ifnextchar <{\Atrianglep}{\Atrianglep
   <\arrowtypea`\arrowtypeb`\arrowtypec;\height>}}
 
\def\Atrianglep<#1>[#2`#3`#4;#5`#6`#7]{{
\settriparms[#1]
\width=\height                         
\xext=\width                           
\yext=\height                          
\topadjust[#2``]
\botadjust[#3`#4`{#7}]
\multiply \xext by2 
\leftsladjust[#3`#2`{#5}]
\rightsladjust[#4`#2`{#6}]
\begin{picture}(\xext,\yext)(\xoff,\yoff)%
\putAtrianglep<\arrowtypea`\arrowtypeb`\arrowtypec;\height>%
(0,0)[#2`#3`#4;#5`#6`{#7}]%
\end{picture}%
}}
 
\def\putAtrianglepairp<#1>(#2)[#3;#4`#5`#6`#7`#8]{{
\settripairparms[#1]%
\setpos(#2)%
\settokens[#3]%
\puthmorphism(\xpos,\ypos)[\tokenb`\tokenc`{#7}]{\height}{\arrowtyped}b%
\advance\xpos by\height
\advance\ypos by\height
\putmorphism(\xpos,\ypos)(-1,-1)[\tokena``{#4}]{\height}{\arrowtypea}l%
\putvmorphism(\xpos,\ypos)[``{#5}]{\height}{\arrowtypeb}m%
\putmorphism(\xpos,\ypos)(1,-1)[``{#6}]{\height}{\arrowtypec}r%
}}
 
\def\putAtrianglepair{\@ifnextchar <{\putAtrianglepairp}{\putAtrianglepairp%
   <\arrowtypea`\arrowtypeb`\arrowtypec`\arrowtyped`\arrowtypee;\height>}}
\def\Atrianglepair{\@ifnextchar <{\Atrianglepairp}{\Atrianglepairp%
   <\arrowtypea`\arrowtypeb`\arrowtypec`\arrowtyped`\arrowtypee;\height>}}
 
\def\Atrianglepairp<#1>[#2;#3`#4`#5`#6`#7]{{%
\settripairparms[#1]%
\settokens[#2]%
\width=\height
\xext=\width
\yext=\height
\topadjust[\tokena``]%
\vertadjust[\tokenb`\tokenc`{#6}]
\tempcountd=\tempcounta                       
\vertadjust[\tokenc`\tokend`{#7}]
\ifnum\tempcounta<\tempcountd                 
\tempcounta=\tempcountd\fi                    
\advance \yext by\tempcounta                  
\advance \yoff by-\tempcounta                 %
\multiply \xext by2 
\leftsladjust[\tokenb`\tokena`{#3}]
\rightsladjust[\tokend`\tokena`{#5}]%
\begin{picture}(\xext,\yext)(\xoff,\yoff)%
\putAtrianglepairp
<\arrowtypea`\arrowtypeb`\arrowtypec`\arrowtyped`\arrowtypee;\height>%
(0,0)[#2;#3`#4`#5`#6`{#7}]%
\end{picture}%
}}
 
\def\putVtrianglep<#1>(#2,#3)[#4`#5`#6;#7`#8`#9]{{%
\settriparms[#1]%
\xpos=#2 \ypos=#3
\advance\ypos by\height
{\multiply\height by2
\puthmorphism(\xpos,\ypos)[#4`#5`{#7}]{\height}{\arrowtypea}a}%
\putmorphism(\xpos,\ypos)(1,-1)[`#6`{#8}]{\height}{\arrowtypeb}l%
\advance\xpos by\height
\advance\xpos by\height
\putmorphism(\xpos,\ypos)(-1,-1)[``{#9}]{\height}{\arrowtypec}r%
}}

\def\putVtriangle{\@ifnextchar <{\putVtrianglep}{\putVtrianglep
   <\arrowtypea`\arrowtypeb`\arrowtypec;\height>}}
\def\Vtriangle{\@ifnextchar <{\Vtrianglep}{\Vtrianglep
   <\arrowtypea`\arrowtypeb`\arrowtypec;\height>}}
 
\def\Vtrianglep<#1>[#2`#3`#4;#5`#6`#7]{{
\settriparms[#1]
\width=\height                         
\xext=\width                           
\yext=\height                          
\topadjust[#2`#3`{#5}]
\botadjust[#4``]
\multiply \xext by2 
\leftsladjust[#2`#3`{#6}]
\rightsladjust[#3`#4`{#7}]
\begin{picture}(\xext,\yext)(\xoff,\yoff)%
\putVtrianglep<\arrowtypea`\arrowtypeb`\arrowtypec;\height>%
(0,0)[#2`#3`#4;#5`#6`{#7}]%
\end{picture}%
}}
 
\def\putVtrianglepairp<#1>(#2)[#3;#4`#5`#6`#7`#8]{{
\settripairparms[#1]%
\setpos(#2)%
\settokens[#3]%
\advance\ypos by\height
\putmorphism(\xpos,\ypos)(1,-1)[`\tokend`{#6}]{\height}{\arrowtypec}l%
\puthmorphism(\xpos,\ypos)[\tokena`\tokenb`{#4}]{\height}{\arrowtypea}a%
\advance\xpos by\height
\putvmorphism(\xpos,\ypos)[``{#7}]{\height}{\arrowtyped}m%
\advance\xpos by\height
\putmorphism(\xpos,\ypos)(-1,-1)[``{#8}]{\height}{\arrowtypee}r%
}}
 
\def\putVtrianglepair{\@ifnextchar <{\putVtrianglepairp}{\putVtrianglepairp%
    <\arrowtypea`\arrowtypeb`\arrowtypec`\arrowtyped`\arrowtypee;\height>}}
\def\Vtrianglepair{\@ifnextchar <{\Vtrianglepairp}{\Vtrianglepairp%
    <\arrowtypea`\arrowtypeb`\arrowtypec`\arrowtyped`\arrowtypee;\height>}}
 
\def\Vtrianglepairp<#1>[#2;#3`#4`#5`#6`#7]{{%
\settripairparms[#1]%
\settokens[#2]
\xext=\height                  
\width=\height                 
\yext=\height                  
\vertadjust[\tokena`\tokenb`{#4}]
\tempcountd=\tempcounta        
\vertadjust[\tokenb`\tokenc`{#5}]
\ifnum\tempcounta<\tempcountd%
\tempcounta=\tempcountd\fi
\advance \yext by\tempcounta
\botadjust[\tokend``]%
\multiply \xext by2
\leftsladjust[\tokena`\tokend`{#6}]%
\rightsladjust[\tokenc`\tokend`{#7}]%
\begin{picture}(\xext,\yext)(\xoff,\yoff)%
\putVtrianglepairp
<\arrowtypea`\arrowtypeb`\arrowtypec`\arrowtyped`\arrowtypee;\height>%
(0,0)[#2;#3`#4`#5`#6`{#7}]%
\end{picture}%
}}

\def\putCtrianglep<#1>(#2,#3)[#4`#5`#6;#7`#8`#9]{{%
\settriparms[#1]%
\xpos=#2 \ypos=#3
\advance\ypos by\height
\putmorphism(\xpos,\ypos)(1,-1)[``{#9}]{\height}{\arrowtypec}l%
\advance\xpos by\height
\advance\ypos by\height
\putmorphism(\xpos,\ypos)(-1,-1)[#4`#5`{#7}]{\height}{\arrowtypea}l%
{\multiply\height by 2
\putvmorphism(\xpos,\ypos)[`#6`{#8}]{\height}{\arrowtypeb}r}%
}}
 
\def\putCtriangle{\@ifnextchar <{\putCtrianglep}{\putCtrianglep
    <\arrowtypea`\arrowtypeb`\arrowtypec;\height>}}
\def\Ctriangle{\@ifnextchar <{\Ctrianglep}{\Ctrianglep
    <\arrowtypea`\arrowtypeb`\arrowtypec;\height>}}
 
\def\Ctrianglep<#1>[#2`#3`#4;#5`#6`#7]{{
\settriparms[#1]
\width=\height                          
\xext=\width                            
\yext=\height                           
\multiply \yext by2 
\topadjust[#2``]
\botadjust[#4``]
\sladjust[#3`#2`{#5}]{\width}
\tempcountd=\tempcounta                 
\sladjust[#3`#4`{#7}]{\width}
\ifnum \tempcounta<\tempcountd          
\tempcounta=\tempcountd\fi              
\advance \xext by\tempcounta            
\advance \xoff by-\tempcounta           %
\rightadjust[#2`#4`{#6}]
\begin{picture}(\xext,\yext)(\xoff,\yoff)%
\putCtrianglep<\arrowtypea`\arrowtypeb`\arrowtypec;\height>%
(0,0)[#2`#3`#4;#5`#6`{#7}]%
\end{picture}%
}}
 
\def\putDtrianglep<#1>(#2,#3)[#4`#5`#6;#7`#8`#9]{{%
\settriparms[#1]%
\xpos=#2 \ypos=#3
\advance\xpos by\height \advance\ypos by\height
\putmorphism(\xpos,\ypos)(-1,-1)[``{#9}]{\height}{\arrowtypec}r%
\advance\xpos by-\height \advance\ypos by\height
\putmorphism(\xpos,\ypos)(1,-1)[`#5`{#8}]{\height}{\arrowtypeb}r%
{\multiply\height by 2
\putvmorphism(\xpos,\ypos)[#4`#6`{#7}]{\height}{\arrowtypea}l}%
}}
 
\def\putDtriangle{\@ifnextchar <{\putDtrianglep}{\putDtrianglep
    <\arrowtypea`\arrowtypeb`\arrowtypec;\height>}}
\def\Dtriangle{\@ifnextchar <{\Dtrianglep}{\Dtrianglep
   <\arrowtypea`\arrowtypeb`\arrowtypec;\height>}}
 
\def\Dtrianglep<#1>[#2`#3`#4;#5`#6`#7]{{
\settriparms[#1]
\width=\height                         
\xext=\height                          
\yext=\height                          
\multiply \yext by2 
\topadjust[#2``]
\botadjust[#4``]
\leftadjust[#2`#4`{#5}]
\sladjust[#3`#2`{#5}]{\height}
\tempcountd=\tempcountd                
\sladjust[#3`#4`{#7}]{\height}
\ifnum \tempcounta<\tempcountd         
\tempcounta=\tempcountd\fi             
\advance \xext by\tempcounta           %
\begin{picture}(\xext,\yext)(\xoff,\yoff)
\putDtrianglep<\arrowtypea`\arrowtypeb`\arrowtypec;\height>%
(0,0)[#2`#3`#4;#5`#6`{#7}]%
\end{picture}%
}}
 
\def\setrecparms[#1`#2]{\width=#1 \height=#2}%
\def\recursep<#1`#2>[#3;#4`#5`#6`#7`#8]{{%
\width=#1 \height=#2
\settokens[#3]
\settowidth{\tempdimen}{$\tokena$}
\ifdim\tempdimen=0pt
  \savebox{\tempboxa}{\hbox{$\tokenb$}}%
  \savebox{\tempboxb}{\hbox{$\tokend$}}%
  \savebox{\tempboxc}{\hbox{$#6$}}%
\else
  \savebox{\tempboxa}{\hbox{$\hbox{$\tokena$}\times\hbox{$\tokenb$}$}}%
  \savebox{\tempboxb}{\hbox{$\hbox{$\tokena$}\times\hbox{$\tokend$}$}}%
  \savebox{\tempboxc}{\hbox{$\hbox{$\tokena$}\times\hbox{$#6$}$}}%
\fi
\ypos=\height
\divide\ypos by 2
\xpos=\ypos
\advance\xpos by \width
\xext=\xpos \yext=\height
\topadjust[#3`\usebox{\tempboxa}`{#4}]%
\botadjust[#5`\usebox{\tempboxb}`{#8}]%
\sladjust[\tokenc`\tokenb`{#5}]{\ypos}%
\tempcountd=\tempcounta
\sladjust[\tokenc`\tokend`{#5}]{\ypos}%
\ifnum \tempcounta<\tempcountd
\tempcounta=\tempcountd\fi
\advance \xext by\tempcounta
\advance \xoff by-\tempcounta
\rightadjust[\usebox{\tempboxa}`\usebox{\tempboxb}`\usebox{\tempboxc}]%
\bfig
\putCtrianglep<-1`1`1;\ypos>(0,0)[`\tokenc`;#5`#6`{#7}]%
\puthmorphism(\ypos,0)[\tokend`\usebox{\tempboxb}`{#8}]{\width}{-1}b%
\puthmorphism(\ypos,\height)[\tokenb`\usebox{\tempboxa}`{#4}]{\width}{-1}a%
\advance\ypos by \width
\putvmorphism(\ypos,\height)[``\usebox{\tempboxc}]{\height}1r%
\efig
}}
 
\def\recurse{\@ifnextchar <{\recursep}{\recursep<\width`\height>}}
 
\def\puttwohmorphisms(#1,#2)[#3`#4;#5`#6]#7#8#9{{%
\puthmorphism(#1,#2)[#3`#4`]{#7}0a
\ypos=#2
\advance\ypos by 20
\puthmorphism(#1,\ypos)[\phantom{#3}`\phantom{#4}`#5]{#7}{#8}a
\advance\ypos by -40
\puthmorphism(#1,\ypos)[\phantom{#3}`\phantom{#4}`#6]{#7}{#9}b
}}
 
\def\puttwovmorphisms(#1,#2)[#3`#4;#5`#6]#7#8#9{{%
\putvmorphism(#1,#2)[#3`#4`]{#7}0a
\xpos=#1
\advance\xpos by -20
\putvmorphism(\xpos,#2)[\phantom{#3}`\phantom{#4}`#5]{#7}{#8}l
\advance\xpos by 40
\putvmorphism(\xpos,#2)[\phantom{#3}`\phantom{#4}`#6]{#7}{#9}r
}}
 
\def\puthcoequalizer(#1)[#2`#3`#4;#5`#6`#7]#8#9{{%
\setpos(#1)%
\puttwohmorphisms(\xpos,\ypos)[#2`#3;#5`#6]{#8}11%
\advance\xpos by #8
\puthmorphism(\xpos,\ypos)[\phantom{#3}`#4`#7]{#8}1{#9}
}}
 
\def\putvcoequalizer(#1)[#2`#3`#4;#5`#6`#7]#8#9{{%
\setpos(#1)%
\puttwovmorphisms(\xpos,\ypos)[#2`#3;#5`#6]{#8}11%
\advance\ypos by -#8
\putvmorphism(\xpos,\ypos)[\phantom{#3}`#4`#7]{#8}1{#9}
}}
 
\def\putthreehmorphisms(#1)[#2`#3;#4`#5`#6]#7(#8)#9{{%
\setpos(#1) \settypes(#8)
\if a#9 %
     \vertsize{\tempcounta}{#5}%
     \vertsize{\tempcountb}{#6}%
     \ifnum \tempcounta<\tempcountb \tempcounta=\tempcountb \fi
\else
     \vertsize{\tempcounta}{#4}%
     \vertsize{\tempcountb}{#5}%
     \ifnum \tempcounta<\tempcountb \tempcounta=\tempcountb \fi
\fi
\advance \tempcounta by 60
\puthmorphism(\xpos,\ypos)[#2`#3`#5]{#7}{\arrowtypeb}{#9}
\advance\ypos by \tempcounta
\puthmorphism(\xpos,\ypos)[\phantom{#2}`\phantom{#3}`#4]{#7}{\arrowtypea}{#9}
\advance\ypos by -\tempcounta \advance\ypos by -\tempcounta
\puthmorphism(\xpos,\ypos)[\phantom{#2}`\phantom{#3}`#6]{#7}{\arrowtypec}{#9}
}}
 
\def\putarc(#1,#2)[#3`#4`#5]#6#7#8{{%
\xpos #1
\ypos #2
\width #6
\arrowlength #6
\putbox(\xpos,\ypos){#3\vphantom{#4}}%
{\advance \xpos by\arrowlength
\putbox(\xpos,\ypos){\vphantom{#3}#4}}%
\horsize{\tempcounta}{#3}%
\horsize{\tempcountb}{#4}%
\divide \tempcounta by2
\divide \tempcountb by2
\advance \tempcounta by30
\advance \tempcountb by30
\advance \xpos by\tempcounta
\advance \arrowlength by-\tempcounta
\advance \arrowlength by-\tempcountb
\halflength=\arrowlength \divide\halflength by 2
\divide\arrowlength by 5
\put(\xpos,\ypos){\bezier{\arrowlength}(0,0)(50,50)(\halflength,50)}
\ifnum #7=-1 \put(\xpos,\ypos){\vector(-3,-2)0} \fi
\advance\xpos by \halflength
\put(\xpos,\ypos){\xpos=\halflength \advance\xpos by -50
   \bezier{\arrowlength}(0,50)(\xpos,50)(\halflength,0)}
\ifnum #7=1 {\advance \xpos by
   \halflength \put(\xpos,\ypos){\vector(3,-2)0}} \fi
\advance\ypos by 50
\vertsize{\tempcounta}{#5}%
\divide\tempcounta by2
\advance \tempcounta by20
\if a#8 %
   \advance \ypos by\tempcounta
   \putbox(\xpos,\ypos){#5}%
\else
   \advance \ypos by-\tempcounta
   \putbox(\xpos,\ypos){#5}%
\fi
}}
 
\makeatother

\input{Macro}

\author[Ghilardi]{Silvio Ghilardi}
\address{Silvio Ghilardi\\
 Dipartimento di Matematica, Universit\`a degli Studi di
  Milano}
\email{silvio.ghilardi@unimi.it}

\author[Santocanale]{Luigi Santocanale}
\address{Luigi Santocanale\\
LIS, CNRS UMR 7020, Aix-Marseille Universit\'e}
\email{luigi.santocanale@lis-lab.fr}

\title[Ruitenburg's Theorem via Duality and Bounded Bisimulations]{Ruitenburg's Theorem \\ via Duality and Bounded Bisimulations}

\begin{document}

\maketitle

\begin{abstract}
  For a given intuitionistic propositional formula $A$ and a
  propositional variable $x$ occurring in it, define the infinite
  sequence of formulae $\set{ A^i}_{i\geq 1}$ by letting $A^1$ be $A$
  and $A^{i+1}$ be $A(A^i/x)$.
  \RT~\cite{Ruitenburg84} says that the sequence
  $\set{ A^i}_{i\geq 1}$ (modulo logical equivalence) is ultimately
  periodic with period 2, i.e.  there is $N\geq 0$ such that
  $A^{N+2}\leftrightarrow A^N$ is provable in intuitionistic
  propositional calculus. We give a semantic proof of this theorem,
  using duality techniques and bounded bisimulations ranks.


\end{abstract}

\section{Introduction}\label{sec:intro}

Let us call an infinite sequence 
$$
a_1, a_2, \dots, a_i, \dots
$$
\emph{ultimately periodic} iff there are $N$ and $k$ such that for all
$s_1, s_2\geq N$, we have that 
$s_1\equiv s_2 \mod k$
implies
$a_{s_1}= a_{s_2}$.  If $(N, k)$ is the smallest (in the lexicographic
sense) pair for which this happens, we say that $N$ is an \emph{index}
and $k$ a \emph{period} for the ultimately periodic sequence
$\set{a_i}_i$. Thus, for instance, an ultimately periodic sequence
with index $N$ and period 2 looks as follows
$$
a_1, \dots, a_N, a_{N+1}, a_N, a_{N+1}, \dots
$$
A typical example of an ultimately periodic sequence is the sequence
of the iterations $\set{f^i}_i$ of an endo-function $f$ of a finite
set.  Whenever infinitary data are involved, ultimate periodicity comes
often as a surprise.

  \RT is in fact a surprising result stating the following: take a
  formula $A(x, \uy)$ of intuitionistic propositional calculus $(IPC)$
  (by the notation $A(x, \uy)$ we mean that the only propositional
  letters occurring in $A$ are among $x, \uy$ - with $\uy$ being, say,
  the tuple $y_1, \dots, y_n$) and consider the sequence
  $\set{A^i(x,\uy)}_{i\geq 1}$ so defined:
\begin{equation}\label{eq:formulasequence}
A^1:\equiv A, ~~\dots,~~ A^{i+1}:\equiv A(A^i/x, \uy)
\end{equation}
where the
  slash means substitution; then,
  \emph{taking equivalence
  classes under provable bi-implication in $(IPC)$, the sequence
  $\set{[A^i(x,\uy)]}_{i\geq 1}$ is ultimately periodic with period
  2}.
  The latter means that there is $N$ such that
\begin{equation}\label{eq:Ruitenburgtheorem}
\vdash_{IPC} A^{N+2}\leftrightarrow  A^N~~~.
\end{equation}

An interesting consequence of this result is that \emph {least (and
  greatest) fixpoints of monotonic formulae are definable in
  $(IPC)$}~\cite{Mardaev1993,Mardaev07,fossacs}: this is because the
sequence~\eqref{eq:formulasequence} becomes increasing when evaluated
on $\bot/x$ (if $A$ is monotonic in $x$), so that the period is
decreased to 1.  Thus the index of the sequence becomes a finite
upper bound for the fixpoints approximation
convergence.

\RT was shown in~\cite{Ruitenburg84} via a, rather involved, purely syntactic proof. The proof has been recently formalized inside the proof assistant \textsc{coq} by T. Litak 
(see \url{https://git8.cs.fau.de/redmine/projects/ruitenburg1984}).
In this paper we supply a semantic proof, using duality 
and bounded bisimulation machinery. 

\emph{Bounded bisimulations}  are a standard tool in non classical logics~\cite{fine} which is used in order to characterize satisfiability of bounded depth formulae
and hence definable classes of models: examples of the use of bounded bisimulations include for instance~\cite{shavrukov},~\cite{GhilardiZawadowski2011},~\cite{visser},~\cite{um}. 

\emph{Duality} has a long tradition in algebraic logic (see e.g.~\cite{Esa74} for the Heyting algebras case): 
many phenomena look more transparent whenever they are analyzed in the
dual categories, especially whenever dualities can convert coproducts
and colimits constructions into more familiar `honest' products and
limits constructions.  The duality we use here is taken
from~\cite{GhilardiZawadowski2011} and has a mixed
geometric/combinatorial nature. In fact, the geometric environment
shows \emph{how to find} relevant mathematical structures (products,
equalizers, images,...) using their standard definitions in sheaves
and presheaves; on the other hand, the combinatorial aspects show that
such constructions \emph{are definable}, thus meaningful from the
logical side. In this sense, notice that we work with finitely
presented algebras, and our combinatoric ingredients
(Ehrenfeucht-Fraiss\'e games, etc.) replace the
topological ingredients which are common in the
algebraic logic literature (working with arbitrary algebras instead).

The paper is organized as follows. In Section~\ref{sec:classical} we show how to formulate \RT in algebraic terms and how to prove it via duality in the easy case of classical logic (where index is always 1). This Section supplies the methodology we shall follow in the whole paper. After introducing  
the required duality ingredients for finitely presented Heyting algebras  (this is done in Section~\ref{sec:duality} - the material of this Section is taken from~\cite{GhilardiZawadowski2011}), we show how to extend the basic argument of Section~\ref{sec:classical} to finite Kripke models in Section~\ref{sec:finite}.
This extension does not directly give \RT, because it supplies a bound for the indexes of our sequences 
which is dependent on the poset a given model is based on. This bound is  made uniform in Section~\ref{sec:main} (using the ranks machinery introduced in Section~\ref{sec:ranks}), thus finally reaching our goal.

\section{The Case of Classical Logic}\label{sec:classical}

We explain our methodology in the much easier case of classical logic. In classical propositional calculus ($CPC$), 
\RT holds with index 1 and period 2, namely
given a formula $A(x,\uy)$, we need to prove that
\begin{equation}
  \label{eq:class}
 \vdash_{CPC} A^3 \leftrightarrow A 
\end{equation}
holds (here $A^3$ is defined like in~\eqref{eq:formulasequence}). 
\subsection{The algebraic reformulation}

First, we transform the above statement \eqref{eq:class}
into an algebraic statement concerning
free Boolean algebras. We let $\cF_B(\uz)$ be the free Boolean algebra
over the finite set $\uz$.  Recall that $\cF_B(\uz)$ is the
Lindenbaum-Tarski algebra of classical propositional calculus
restricted to a language having just the $\uz$ as propositional
variables.

Similarly, morphisms $\mu: \cF_B(x_1, \dots, x_n)\lora \cF_B(\uz)$
bijectively correspond to $n$-tuples of equivalence classes of
formulae $A_1(\uz), \dots, A_n(\uz)$ in $\cF_B(\uz)$: the map $\mu$
corresponding to the tuple $A_1(\uz), \dots, A_n(\uz)$ associates with
the equivalence class of $B(x_1, \dots, x_n)$ in
$\cF_B(x_1, \dots, x_n)$ the equivalence class of
$B(A_1/x_1, \dots, A_n/x_n)$ in $\cF_B(\uz)$.

Composition is substitution, in the sense that 
if  $\mu: \cF_B(x_1, \dots, x_n)\lora \cF_B(\uz)$ is induced, as above, by $A_1(\uz), \dots, A_n(\uz)$ and if $\nu:\cF_B(y_1, \dots, y_m)
\lora \cF_B(x_1, \dots, x_n)$ is induced by $C_1(x_1, \dots, x_n), \dots, C_m(x_1, \dots, x_n)$, then the composite map 
$\mu\circ \nu: \cF_B(y_1, \dots, y_m)\lora \cF_B(\uz)$ is induced by the $m$-tuple $C_1(A_1/x_1, \dots, A_n/x_n), \dots, C_m(A_1/x_1, \dots, A_n/x_n)$.

How to translate the statement~\eqref{eq:class} in this setting? Let
$\uy$ be $y_1, \dots, y_n$; we can consider the map
$\mu_A:\cF_B(x,y_1, \dots, y_n)\lora \cF_B(x,y_1, \dots, y_n)$ induced
by the $n+1$-tuple of formulae $A, y_1, \dots, y_n$; then, taking in
mind that in Lindenbaum algebras identity is modulo provable
equivalence, the statement~\eqref{eq:class} is equivalent to
\begin{equation}\label{eq:class1}
 \mu_A^3 = \mu_A~~. 
\end{equation}
This raises the question: which endomorphisms of $\cF_B(x,\uy)$ are of the kind $\mu_A$ for some $A(x, \uy)$? The answer is simple: consider the `inclusion'
map $\iota$ of  $\cF_B(\uy)$ into $\cF_B(x,\uy)$ (this is the map induced by the $n$-tuple $y_1, \dots, y_n$): the maps 
$\mu:\cF_B(x,\uy)\lora\cF_B(x,\uy) $
that are of the kind $\mu_A$ are precisely the maps $\mu$ such that $\mu\circ \iota= \iota$, i.e. those for which the triangle
\begin{center}
\resetparms
\settriparms[1`1`1;400] \Atriangle[\cF_B(\uy)`\cF_B(x,\uy) 
`\cF_B(x,\uy) ;\iota`\iota`\mu ]
\end{center}
\noindent
commutes. 

It is worth making a little step further: since the free algebra
functor preserves coproducts, we have that $\cF_B(x,\uy)$ is the
coproduct of $\cF_B(\uy)$ with $\cF_B(x)$ - the latter being the free
algebra on one generator. In general, let us denote by $\cA[x]$ the
coproduct of the Boolean algebra $\cA$ with the free algebra on one
generator (let us call $\cA[x]$ the \emph{algebra of polynomials} over
$\cA$).

A  slight generalization of statement~\eqref{eq:class1} now reads as follows:

\begin{itemize}
\item let $\cA$ be a finitely presented Boolean
  algebra\footnote{Recall that an algebra is finitely presented iff it
    is isomorphic to the quotient of a finitely generated free algebra
    by a finitely generated congruence. In the case of Boolean algebra
    `finitely presented' is the same as `finite', but it is not
    anymore like that in the case of Heyting algebras.} and let the
  map $\mu: \cA[x]\lora \cA[x]$ commute with the coproduct injection
  $\iota: \cA \lora \cA[x]$
\begin{center}
\resetparms
\settriparms[1`1`1;400] \Atriangle[\cA`\cA[x] 
`\cA[x] ;\iota`\iota`\mu ]
\end{center}
\noindent
Then we have 
\begin{equation}\label{eq:class2}
\mu^3=\mu~~. 
\end{equation}
\end{itemize}

\subsection{Duality}

The gain we achieved with statement~\eqref{eq:class2} is that the latter is a purely categorical statement, so that we can re-interpret it in dual categories.
In fact, a good duality may turn coproducts into products and make our statement easier - if not trivial at all.

Finitely presented Boolean algebras are dual to finite sets; the duality functor maps coproducts into products and the free Boolean algebra on one generator 
to the two-elements set ${\bf 2}=\set{0,1}$ (which, by chance is also a subobject classifier for finite sets). Thus statement~\eqref{eq:class2} now becomes
\begin{itemize}
\item let $T$ be a finite set and let the function
  $f: T\times {\bf{2}}\lora T\times {\bf{2}}$ commute with the product
  projection $\pi_0: T\times {\bf{2}} \lora T$
  \begin{center}
    \resetparms \settriparms[1`1`1;400] \Vtriangle[ T\times {\bf{2}}`
    T\times {\bf{2}} `T ;f`\pi_0`\pi_0 ]
\end{center}
\noindent
Then we have 
\begin{equation}\label{eq:class3}
f^3=f~~. 
\end{equation}
\end{itemize}

In this final form, statement~\eqref{eq:class3} is now just a trivial exercise, which is solved as follows. Notice first that $f$ can be decomposed as 
$\langle \pi_0, \chi_S\rangle$ (incidentally, $\chi_S$ is the characteristic function of some $S\subseteq T\times \bf{2}$).
Now, if $f(a,b)=(a,b)$ we trivially have also $f^3(a,b)=f(a,b)$; suppose then
 $f(a,b)=(a, b')\neq (a,b)$.
 If $f(a,b')=(a,b')$, then $f^3(a,b)=f(a,b)=(a,b')$, otherwise
 $f(a,b')=(a,b)$ (there are only two available values for $b$!) and
 even in this case $f^3(a,b)=f(a,b)$.

 Let us illustrate theses cases by thinking of the action of $f$ on
 $A\times {\bf 2}$ as one-letter deterministic automaton:
 
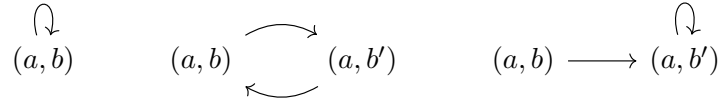
\begin{center}
   \begin{tikzcd}
     (a,b)\ar[loop above]{}{} &
     (a,b) \ar[r, bend left]{}{} & (a,b') \ar[l,, bend left]{}{}&
     (a,b) \ar[r]{}{} & (a,b') \ar[loop above]{}{}
  \end{tikzcd}
 \end{center}
 This means that on each irreducible component of the action the pairs
 index/period are among $(0,1)$, $(0,2)$, $(1,1)$. Out of these pairs
 we can compute the global index/period of $f$ by means of a
 $\max/\lcm$ formula:
 $(1,2) = (\max\set{0,0,1},\lcm\set{1,2})$.

\section{Duality for Heyting Algebras}
\label{sec:duality}

In this Section we supply definitions, notation and statements
from~\cite{GhilardiZawadowski2011} concerning duality for \fp Heyting
algebras.  Proofs of the facts stated in this section can all be found
in~\cite[Chapter
  4]{GhilardiZawadowski2011}.

  A partially ordered set (poset, for short) is a set endowed with a
  reflexive, transitive, antisymmetric relation (to be always denoted
  with $\leq$). 
  A poset $P$ is rooted if it has a greatest element, that we shall
  denote by $\Rho{P}$. 
If a finite poset $L$ is fixed, we call an $L$-{\it evaluation} or
simply an {\it evaluation} a pair $\langle P, u\rangle$, where $P$ is
a rooted finite poset and $u:P\ra L$ is an order-preserving map.
 
Evaluations \emph{restrictions} are introduced as follows.
If $\langle P, u\rangle$ is an $L$-evaluation and if $p \in P$,
  then we shall denote by $u_{p}$ the  $L$-evaluation,
  $\langle \downset p, u \circ i\rangle$, where
  $\downset p = \set{p ' \in P \mid p' \leq p}$ and
  $i : \,\downset p \subseteq P$ is the inclusion map; briefly,
  $u_{p}$ is the restriction of $u$ to the downset generated by $p$.

Evaluations have a
strict relationship with finite Kripke models: we show in detail  
the connection. 
If
$\langle L, \leq\rangle$ is $\langle {\cal P}(\ux), \supseteq\rangle$
(where $\ux=x_1, \dots, x_n$ is a finite list of propositional letters), then an
$L$-evaluation $u:P\ra L$ is called  a \emph{Kripke model} for the
propositional intuitionistic language built up from $\ux$.\footnote{
  According to our conventions, we have that (for $p, q\in P$) if
  $p\leq q$ then $u(p)\supseteq u(q)$, that is we use $\leq$ where
  standard literature uses $\geq$.}  
  Given such a Kripke model $u$ and an IPC formula $A(\ux)$, the \emph{forcing} relation 
  $u\models A$ is inductively defined as follows:
$$
\begin{aligned}
 &u\models x_i~ &{\rm iff}~ &x_i\in u(\Rho{P})
  \\
 &u\not \models \bot &&
 \\
 &u\models A_1\wedge A_2~&{\rm iff}~&(u\models A_1~{\rm and}~u\models A_2)
 \\
 &u\models A_1\vee A_2~&{\rm iff}~&(u\models A_1~{\rm or}~u\models A_2)
 \\
 &u\models A_1\to A_2~&{\rm iff}~&\forall q\leq \Rho{P}~(u_q\models A_1~{\Rightarrow}~u_q\models A_2)~~.
\end{aligned}
$$

We define for every $n\in \omega$ and for every pair of
$L$-evaluations $u$ and $v$, the notions of being {\it $n$-equivalent}
(written $u\sim_n v$). 
We also define, for two
  $L$-evaluations $u, v$, the notions of being {\it infinitely
    equivalent} (written $u\sim_{\infty}v$).

Let $u:P\ra L$ and $v:Q\ra L$ be two $L$-evaluations. 
The {\it game} we are interested in has two
  players, Player 1 and Player 2.  
Player 1 can choose either a point in $P$ or a point in $Q$ and Player
2 must answer by choosing a point in the other poset; the only rule of the game is that, if
$\langle p\in P, q\in Q\rangle$ is the last move played so far, then
in the successive move the two players can only choose points
$\langle p', q'\rangle$ such that $p'\leq p$ and $q'\leq q$. If
$\langle p_1, q_1\rangle, \dots, \langle p_i, q_i\rangle, \dots$ are
the points chosen in the game, Player 2 wins iff for every
$i=1, 2, \dots$, we have that $u(p_i)=v(q_i)$. We say that
\begin{itemize}
\item[-]
$u\sim_{\infty} v$ iff {\it Player 2 has a winning strategy} in the above game with infinitely many moves;
\item[-]
$u\sim_n v$ (for $n>0$) iff {\it Player 2 has a winning strategy}
in the above game with $n$ moves, i.e. he has a winning strategy provided
we stipulate that the game terminates after $n$ moves;
\item[-] $u\sim_0 v$ iff $u(\Rho{P})=v(\Rho{Q})$
  (recall that $\Rho{P}, \Rho{Q}$ denote the roots
  of $P, Q$).
\end{itemize}

Notice that $u\sim_n v$ always implies $u\sim_0 v$, by the fact that
$L$-evaluations are order-preserving. 
We shall use the notation $[v]_{n}$ for the equivalence class of
  an $L$-valuation $v$ via the equivalence relation $\sim_{n}$.

 The following Proposition states a basic fact (keeping the above
definition for $\sim_0$ as base case for recursion, the Proposition
also supplies an alternative recursive
definition for $\sim_n$):
 
\begin{proposition}\label{p41.1}  
 Given two $L$-evaluations
$u:P\ra L, v:Q\ra L$, and $n>0$, we have that
$u\sim_{n+1}v$ iff $\forall p\in P\;\exists q\in Q ~(u_p\sim_n
v_q)$  and vice versa.
\end{proposition}

It can be shown that in case $L={\cal P}(x_1, \dots, x_n)$ (i.e. when
$L$-evaluations are just ordinary finite Kripke models over the
language built up from the propositional variables $x_1, \dots, x_n$),
two evaluations are $\sim_{\infty}$-equivalent (resp.
$\sim_n$-equivalent) iff they force 
the same formulas (resp. the same formulas up to implicational degree
$n$). This can be explained in a formal way as follows. For an IPC formula $A(\ux)$, define the implicational degree $d(A)$ as follows:ù
\begin{description}
 \item[{\rm (i)}] $d(\bot)=d(x_i)=0$, for $x_i\in \ux$;
 \item[{\rm (ii)}] $d(A_1*A_2)= max[d(A_1), d(A_2)]$, for $*=\wedge, \vee$;
 \item[{\rm (iii)}] $d(A_1\to A_2)= max[d(A_1), d(A_2)]+1$.
\end{description}
Then one can prove~\cite{visser} that: (1) $u\sim_\infty v$ holds precisely iff ($u\models A \Leftrightarrow v\models A$) holds for all formulae $A(\ux)$;
(2) for all $n$, $u\sim_n v$ holds precisely iff ($u\models A \Leftrightarrow v\models A$) holds for all formulae $A(\ux)$ with $d(A)\leq n$.\footnote{For (1) to be true, it is essential our evaluations to be defined over \emph{finite} posets.}

The above discussion motivates a sort of  identification of formulae with sets of evaluations closed under restrictions and under $\sim_n$ for some $n$. Thus,
\emph{bounded bisimulations} (this is the way the relations $\sim_n$
are sometimes called) supply the combinatorial ingredients for our
duality; for the picture to be complete, however, we also need a geometric environment, which we 
introduce using presheaves.

A map among posets is said to be {\it open}\footnote{Open surjective
  maps are called p-morphisms in the standard non classical logics
  terminology.}  iff it is open in the topological sense (posets can
be viewed as topological spaces whose open subsets are the downward
closed subsets); thus $f : Q\lora P$ is open iff it
is order-preserving and moreover satisfies the following condition
forall $q\in Q, p\in P$
$$
 p\leq f(q) ~\Rightarrow~ \exists q'\in Q~(q'\leq q ~\&~f(q')=p)~~. 
$$
Let $\Pzero$ be the category of finite rooted posets and open maps
between them; 
  a \emph{presheaf} over $\Pzero$ is a contravariant
  functor from $\Pzero$ to the category of sets and function, that is,
  a functor $H:\Pzero^{op}\lora \bf Set$.
Let us recall what this means: a functor $H:\Pzero^{op}\lora \bf Set$
associates to each finite rooted poset $P$ a set $\FH{P}$; if
$f : Q \lora P$ is an open map, then we are also given a function
$\FH{f} : \FH{P} \lora \FH{Q}$; moreover, identities are sent to
identities, while composition is reversed,
$H(g \circ f) = H(f) \circ H(g)$.

Our presheaves form a category 
whose objects are prersheaves over $\Pzero$ and whose maps are
natural transformations; recall that a natural transformation
$\psi: H\lora H'$ is a collections of maps
$\psi_P:\FH{P}\lora \FH[H']{P}$ (indexed by the objects of $\Pzero$)
such that for every map $f:Q\lora P$ in $\Pzero$, we have
$\FH[H']{f}\circ \psi_P = \psi_Q\circ
\FH{f}$. Throughout the paper, we shall usually
omit the subscript $P$ when referring to the $P$-component $\psi_P$ of
a natural transformation $\psi$.

The basic example of presheaf we need in the paper is described as follows.
 Let $L$ be a finite poset and let $h_L$ be the contravariant functor so defined:
 \begin{itemize}
  \item for a finite poset $P$, $h_L(P)$ is the set of all $L$-evaluations;
  \item for an open map $f:Q\lora P$, $h_L(f)$ takes $v: P\lora L$ to
    $v\circ f: Q\lora L$.
 \end{itemize}
The presheaf $h_L$ is actually a sheaf (for the canonical Grothendieck topology over $\Pzero$); we won't need this fact,\footnote{ 
The sheaf structure becomes essential for instance when one has to compute images - images are the categorical counterparts of second order quantifiers, 
see~\cite{GhilardiZawadowski2011}.
}
but 
we nevertheless call $h_L$ the \emph{sheaf of $L$-evaluations} (presheaves of the kind $h_L$, for some $L$, are called \emph{evaluation sheaves}).

Notice the following fact: if $\psi: h_L \lora h_{L'}$ is a natural
transformation, $v\in h_L(P)$ and $p\in P$, then
$\psi(v_p)= (\psi(v))_p$ (this is due to the fact that the inclusion
$\downarrow p \subseteq P$ is an open map, hence an arrow in
$\Pzero$); thus, we shall feel free to use the (non-ambiguous)
notation $\psi(v)_p$ to denote $\psi(v_p)= (\psi(v))_p$.

The notion of \emph{bounded bisimulation index} (\emph{b-index}, for short)\footnote{ This is called 'index' tout court in~\cite{GhilardiZawadowski2011};
here we used the word `index' for a different notion, since Section~\ref{sec:intro}.
}  takes together structural and combinatorial aspects. 
We say that a natural transformation $\psi: h_L \lora h_{L'}$ \emph{has b-index $n$} iff for every $v:P\lora L$ and $ v': P'\lora L$, we have that 
$v\sim_n v'$ implies $\psi(v)\sim_0 \psi(v')$.

The following Proposition lists basic facts about b-indexes (in particular, it ensures that natural transformations having a b-index do compose):

\begin{proposition}
  \label{prop:indexDecreases}
  Let $\psi: h_L \lora h_{L'}$ have b-index $n$; then it has also
  b-index $m$ for every $m\geq n$. Moreover, for every $k\geq 0$, for
  every $v:P\lora L$ and $ v': P'\lora L$, we have that
  $v\sim_{n+k} v'$ implies $\psi(v)\sim_k \psi(v')$.
\end{proposition}

We are now ready to state duality theorems. As it is evident from the
discussion in Section~\ref{sec:classical}, it is sufficient to state a
duality for the category of finitely generated free Heyting algebras;
although it would not be difficult to give a duality for finitely
presented Heyting algebras, we just state a duality for the
intermediate category of Heyting algebras freely generated by a finite
bounded distributive lattice (this is quite simple to state and is
sufficient for proving \RT).

 \begin{theorem}\label{thm:duality}
   The category of Heyting algebras freely generated by a finite
   bounded distributive lattice is dual to the subcategory of
   presheaves over $\Pzero$ having as objects the evaluations sheaves
   and as arrows the natural transformations having a b-index.
 \end{theorem}
 
 It is important to notice that in the subcategory mentioned in the
 above Theorem, products are computed as in the category of
 presheaves. This means that they are computed pointwise, like in the category of sets:
 in other words, we have that  $(h_L \times h_{L'})(P) = h_L(P) \times h_{L'}(P)$
 and $(h_L \times h_{L'})(f) = h_L(f) \times h_{L'}(f)$, for all $P$ and $f$.
  Notice moreover
   that $h_{L\times L'}(P)\simeq h_L(P) \times h_{L'}(P)$, so we have
   $h_{L\times L'}\simeq h_L \times h_{L'}$; in addition, 
  the two product
 projections  have b-index 0.
   The situation strongly contrasts with other kind of dualities, 
   see \cite{Esa74} for example, 
   for which products are difficult to compute. 
   The ease by which products are computed
   might be seen as the principal reason for tackling a proof of
   \RT by means of sheaf duality.

  As a final information, we need to identify the dual of the free Heyting algebra on one generator:
  
  \begin{proposition}
   The dual of the free Heyting algebra on one generator is $h_{\bf 2}$, where $\bf 2$ is the two-element poset $\set{ 0, 1}$ with $1\leq 0$.
  \end{proposition}

\section{Indexes and Periods over Finite Models}\label{sec:finite}

Taking into consideration the algebraic reformulation from Section~\ref{sec:classical} and the  information from the previous section, we can prove \RT
for $(IPC)$
by showing that \emph{all natural 
transformations from $h_{L} \times h_{\bf 2}$ into itself, commuting  
over the first projection $\pi_0$ and having a b-index, are ultimately periodic with period 2}.
Spelling this out, this means the following. Fix a finite poset $L$ and a natural transformation $\psi: h_{L} \times h_{\bf 2}\lora h_{L} \times h_{\bf 2}$ having a b-index such that the diagram
  \begin{center}
\resetparms
\settriparms[1`1`1;400] \Vtriangle[ h_{L}\times h_{\bf 2}` h_{L}\times h_{\bf 2} 
`h_{L} ;\psi`\pi_0`\pi_0 ]
\end{center}
\noindent
commutes; we have to find an $N$ such that $\psi^{N+2}=\psi^N$,
according to the dual reformulation of~\eqref{eq:Ruitenburgtheorem}.

From the commutativity of the above triangle, we can decompose $\psi$
as $\psi = \langle \pi_{0},\chi\rangle$, were both
$\pi_0 : h_{L} \times h_{\bf 2} \lora h_{L}$ and
$\chi : h_{L} \times h_{\bf 2} \lora h_{\bf
    2}$
  have a b-index; we assume that $n\geq 1$ is a b-index for both of
  them. \emph{We let such $\psi = \langle \pi_{0},\chi\rangle$ and $n$
    be fixed for the rest of the paper.}

Notice that for $\vu \in h_{L}(P)\times  h_{\bf 2}(P)$, we have
$$
{\psi^{k}}\vu= (v,u_k)
$$
where we put 
\begin{equation}\label{eq:def_un}
 u_0\eqdef u~~{\rm  and}~~ u_{{k}+1}\eqdef\chi(v,
 u_{{k}})
 \,.
\end{equation}
Since $P$  and $L$ are finite, it is clear that the sequence $\set{\psi^{k}\vu \mid k \geq 0}$ 
(and obviously also the sequence $\set{u_{k}\mid k \geq 0}$) must become ultimately periodic.

We show in this section that, for each finite set $P$ and for each
$\vu \in h_{L}(P)$, the period of the sequence
$\set{\psi^{k}\vu \mid k \geq 0}$ has $2$ as an upper bound, whereas
the index of $\set{\psi^{k}\vu \mid k \geq 0}$ can be bounded by the
maximum length of the chains in the finite poset $P$ (in the next
section, we shall bound such an index independently on $P$, thus
proving \RT).

Call $\vu \in h_{L}(P)$ \emph{2-periodic} (or just
\emph{periodic}\footnote{From now on, `periodic' will mean
  `2-periodic', i.e. `periodic with period 2'.}) iff we have
$\psi^2\vu=\vu$; a point $q\in P$ is similarly said periodic in
  $\vu$ iff $\vu_q$ is periodic. We shall only say that $p$ is
  periodic if an evaluation is given and understood from the
  context. We call a point \emph{\nonperiodic} if it is not periodic
  (w.r.t. a given evaluation).
  
\begin{lemma}\label{lem:period}
  Let $\vu \in h_{L}(P)$ and $p\in P$ be such that all $q\in P$,
  $q< p$, are periodic. Then either $\vu_p$ is periodic or $\psi\vu_p$
  is periodic.  Moreover, if $\vu_p$ is \nonperiodic and
  $u_0(p)=u(p)=1$, then $u_1(p)=\chi(u,v)(p)=0$.
\end{lemma}

\begin{proof} We work by induction on the height of $p$ (i.e. on the maximum $\leq$-chain starting with $p$ in $P$).
  If the height of $p$ is $1$, then the argument is the same as in the
  classical logic case (see Section~\ref{sec:classical}).

  If the height is greater than one, then we need a simple
  combinatorial check about the possible cases that might arise.
  Recalling the above definition~\eqref{eq:def_un} of the
  $\bf 2$-evaluations
  $u_n$, 
  the induction hypothesis tells us that there
  is $M$  big enough so that so for all $k \geq M$ and
  $q < p$, $(u_{k+2})_{q} = (u_{k})_{q}$.

Let $\ddownset p = \set{q \in P \mid q < p}$.  We shall represent
$(u_{k})_{p}$ as a pair $\typeOfP{a_{k}}{x_{k}}$, where
$a_{k} = u_{k}(p)$ and $x_{k}$ is the restriction of
$(u_{k})_{p}$ to $\ddownset p$.
  
Let us start by considering a first repeat $(i,j)$ of the sequence
$\set{a_{M +k}}_{k\geq 0}$ - that is $i$ is the smallest $i$ such that
there is $j > 0$ such that $a_{M +i+j}= a_{M +i}$
and $j$ is the smallest such $j$. Since the $a_{M +n}$ can only take
value 0 or 1, we must have $i+j \leq 2$.  We show that the sequence
$\set{(u_{M +k})_p}_{k\geq 0}$ has first repeat taken from
  $$
  (0,1), (0,2), (1,1), (1,2)\,.ùì
  $$ 
This shall imply in the first two cases
  that $\vu_p$ is periodic or, in the last two cases, that $\psi\vu_p$
  is periodic.  To our goal, let $x = x_{M}$ and $y = x_{M +1}$
(recall that we do now know whether $x = y$).

  Notice that, if $j = 2$, then $i = 0$ and a first repeat for
  $\set{(u_{k})_p}_{k \geq M}$, is $(0,2)$, as in the diagram below
  \begin{align*}
    \typeOfP{a}{x}\typeOfP{b}{y}\typeOfP{a}{x}\,.
  \end{align*}

  Therefore, let us assume $j = 1$ (so $i \in \set{0,1}$). Consider
  firstly $i = 0$:
  \begin{align*}
    \typeOfP{a}{x}\typeOfP{a}{y}\typeOfP{c}{x}\typeOfP{d}{y}
  \end{align*}
  If $x = y$, then we have a repeat at $(0,1)$. Also, if $a =1$, then
  the mappings
    $x$ and $y$ are uniformely $1$,\footnote{Recall that our
    evaluations are order-preserving maps and we have $1\leq 0$ in
    $\bf 2$.}  
  so again $x = y$
  and $(0,1)$ is a repeat.

  So let us assume $x \neq y$ and $a = 0$. If $c = a$, then we have
  the repeat $(0,2)$ as above.  Otherwise $c = 1$, so $x = 1$.  We
  cannot have $d = 1$, otherwise $1 = x = y$. Thus
  $d = 0 = a$, and the repeat is $(1,2)$.

  Finally, consider $i = 1$ (so $a \neq b$ and $j = 1$):
  \begin{align*}
    \typeOfP{a}{x}\typeOfP{b}{y}\typeOfP{b}{x}\typeOfP{d}{y}
  \end{align*}
  We have two  subcases: $b=1$ and $b=0$. If $b=1$, then $a=0$ and $x = 1 = y$: we have a repeat at $(1,1)$.
  
  In the last subcase, we have $b=0$, $a=1$ and now if $d=0$ we have a
  repeat at $(1,2)$ and if $d=1$ we have a repeat $(1,1)$ (because
  $d = a = 1$ implies $y = 1$ and $x =1$).

  The last statement of the Lemma is also obvious in view
    of the fact that if $a = b = 1$, then $x = y =1$, so $p$ is
    periodic.
\end{proof}

\begin{corollary}
 Let $N_P$ be the height of $P$; then $\psi^{N_P}\vu$ is periodic for all  $\vu \in h_{L}(P)$.
\end{corollary}

\begin{proof}
 An easy induction on $N_P$, based on the previous Lemma.
\end{proof}

\section{Ranks}\label{sec:ranks}

Ranks (already introduced in~\cite{fine}) are a powerful tool suggested by bounded bisimulations; in our context the useful notion of rank is given below.
Recall that $\psi = \langle \pi_{0},\chi\rangle$ and that $n\geq 1$ is a b-index for $\psi$ and $\chi$.

Let $\vu \in h_{L}(P)$ be given.  The \emph{type} of a periodic point
$p\in P$ is  the pair of equivalence classes
\begin{equation}\label{eq:pairs}
 \langle [(v_p,u_p)]_{n-1}, [\psi(v_p,u_p)]_{n-1}\rangle.
\end{equation}
The \emph{rank} of a point $p$ (that we shall denote by
  $rk(p)$) is the cardinality of the set of distinct types of the
periodic points $q\leq p$.  Since $\sim_{n-1}$ is an equivalence
relation with finitely many equivalence classes, the rank cannot
exceed a positive number $R(L,n)$ (that can be computed in function of
$L, n$).

Clearly we have $rk(p)\geq rk(q)$ in case $p\geq q$.  Notice that an
application of $\psi$ does not decrease the rank of a point: this is
because the pairs~\eqref{eq:pairs} coming from a periodic point just
get swapped after applying $\psi$.  A non-periodic point $p\in P$ has
\emph{minimal rank} iff we have $rk(p)=rk(q)$ for all \nonperiodic
$q\leq p$.

\begin{lemma}\label{lem:minrank}
  Let $p\in P$ be a \nonperiodic point of minimal rank in
  $\vu \in h_{L}(P)$; suppose also that $\vu$ is constant
    on the set of all \nonperiodic points in
    $\downarrow p$.  Then we have
  $\psi^m\vu_{q_0}\sim_n\psi^m\vu_{q_1}$ for all $m\geq 0$ and for all
  \nonperiodic points $q_0, q_1\leq p$.
\end{lemma}
\begin{proof} We let $\Pi$ be the set of periodic points of $\vu$ that
  are in $\downarrow p$ and let $\Pi^c$ be
  $(\downarrow p)\setminus \Pi$ .  Let us first observe that for every
  $r\in \Pi^c$, we have
\begin{eqnarray*}
\smalleql{\{\langle 
 [(v_s,u_s)]_{n-1}, [\psi(v_s,u_s)]_{n-1}\rangle \mid s\leq r, ~s~\hbox{is  periodic}\}}
& \\
& =~~  \{\langle [(v_s,u_s)]_{n-1}, [\psi(v_s,u_s)]_{n-1}\rangle \mid s\leq p, ~s~\hbox{is periodic}\}
\end{eqnarray*}
(indeed the inclusion $\subseteq$ is because $r\leq p$ and the
inclusion $\supseteq$ is by the minimality of the rank of $p$).
Saying this in words, we have that ``for every periodic $s\leq p$
there is a periodic $s'\leq r$ such that
$(v_s,u_s)\sim_{n-1} (v_{s'},u_{s'})$ and
$\psi(v_s,u_s)\sim_{n-1} \psi(v_{s'},u_{s'})$''; also (by the
definition of 2-periodicity), ``for all $m\geq 0$, for every periodic
$s\leq p$ there is a periodic $s'\leq r$ such that
$\psi^m(v_s,u_s)\sim_{n-1} \psi^m(v_{s'},u_{s'})$''.
By letting both $q_0, q_1$ playing the role of
$r$, we get:
\begin{fact*}
  For every $m\geq 0$, for every
  $q_0, q_1\in \Pi^c$, for every periodic $s\leq q_0$ there is a
  periodic $s'\leq q_0$ such that
  $\psi^m(v_s,u_s)\sim_{n-1}
    \psi^m(v_{s'},u_{s'})$ (and vice versa).
\end{fact*}

We now prove the statement of the theorem by induction on $m$; take
two points $q_0, q_1\in \Pi^c$.

For $m=0$, $\vu_{q_0}\sim_n\vu_{q_1}$ is established as follows: as
long as Player 1  plays in $\Pi^c$, we know $\vu$ is constant so that
Player 2 can answer with an identical move still staying within $\Pi^c$;
as soon as it plays in $\Pi$, Player 2 uses the above Fact to win the game.

The inductive case
  $\psi^{m+1}\vu_{q_0}\sim_n\psi^{m+1}\vu_{q_1}$ is proved in the same
  way, using the Fact (which holds for the integer $m +1$) and
  observing that $\psi^{m+1}$ is constant on $\Pi^{c}$. The latter
  statement can be verified as follows: by the induction hypothesis we
  have $\psi^m\vu_{q}\sim_n\psi^m\vu_{q'}$, so we derive from
  Proposition~\ref{prop:indexDecreases}
  $\psi^{m+1}\vu_{q}\sim_0\psi^{m+1}\vu_{q'}$, for all
  $q, q'\in \Pi^c$; that is, $\psi^{m+1}$ is constant on $\Pi^{c}$. 
\end{proof}

\section{\RT}\label{sec:main}

We can finally prove: 

\begin{theorem} [\RT for IPC]\label{thm:main}
There is $N\geq 1$ such that we have $\psi^{N+2}=\psi^N$.
\end{theorem}

\begin{proof} Let $L$ be a finite poset and let $R\eqdef R(L,n)$ be
  the maximum rank for $n,L$ (see the previous section).  Below, for
  $e\in L$, we let $\vert e\vert$ be the height of $e$ in $L$,
  i.e. the maximum size of chains in $L$ whose maximum element is
  $e$; we let also $\vert L \vert$ be the maximum
  size of a chain in $L$.  We make an induction on
  natural numbers $l\geq 1$ and show the following: \emph{(for each
  $l \geq 1$) there is $N(l)$ such that for every $(v,u)$ and
  $p\in dom(v,u)$ such that $l\geq \vert v(p)\vert$, we have that
  $\psi^{N(l)}(v_p, u_p)$ is periodic.} Once this is proved,
  the statement
    of the Theorem shall be proved with
    $N= N(\vert L\vert)$.\footnote{It will turn out that $N(l)$ is
    $2R(l-1)+1$. }

  If $l=1$, it is easily seen that we can put $N(l)=1$ (this case is
  essentially the classical logic case).

  Pick a $p$ with $\vert v(p)\vert =l>1$; let $N_0$ be the maximum of
  the values $N(l_0)$ for $l_0< l$:\footnote{It is easily seen that we
    indeed have $N_{0} = N(l-1)$. }
    we show that we can take $N(l)$ to be $N_0+2R$.

    Firstly, let $(v, u_0) := \psi^{N_0}\vu$ so all $q$ with
    $\vert v(q)\vert < l$ are periodic in $(v,u_0)$.  After such iterations,
    suppose that $p$ is not yet periodic in $(v, u_0)$.  We let $r$ be
    the minimum rank of points $q\leq p$ which are not
    periodic 
    (all such points $q$ must be such that $v(q)=v(p)$); we show that
    after \emph{two iterations} of $\chi$, all points $p_0\leq p$
    having rank $r$ become periodic or increase their rank, thus
    causing the overall minimum rank below $p$ to increase: this means
    that after at most $2(R-r)\leq 2R$ iterations of $\psi$, all
    points below $p$ ($p$ itself included!)  become periodic
  (otherwise said, we take $R-r$ as the secondary parameter of our double induction).

Pick $p_0\leq p$ having minimal rank $r$;
 thus we have that all $q\leq p_0$
in $(v, u_0)$ are now either periodic or have the same rank and the same $v$-value  as $p_0$
(by the choice of $N_0$ above). Let us divide the  points of $\downarrow p_0$ into four subsets:
\begin{align*}
\Eper~ := &~ \set{ q \mid ~q ~{\rm is~periodic}}\\
E_0~~~\, := &~ \set{ q \mid~q\not\in \Eper~\&~\forall q'\leq q ~(q'\not\in \Eper \Rightarrow u_0(q')=0)} \\
E_1~~~\, := &~ \set{ q \mid~q\not\in \Eper~\&~\forall q'\leq q ~(q'\not\in \Eper \Rightarrow u_0(q')=1)}\\
E_{01}~~\, := &~ \set{ q \mid~ q'\not\in \Eper\cup E_1\cup E_0}\,.
\end{align*}
Let us define \emph{frontier point} a \nonperiodic point $f\leq p$ such that
all $q<f$ are periodic (clearly, a frontier point belongs to
$E_0\cup E_1$); by Lemma~\ref{lem:period}, all frontier points become
periodic after applying $\psi$.  Take a point
$q\in E_i$ and a frontier point $f$ below it;
 since $q$ also has
minimal rank and the hypotheses of Lemma~\ref{lem:minrank} are
satisfied for $\vu_q$, we have in particular that
$\psi^m(v,u_0)_{q'}=\psi^m(v,u_0)_f$
for all $m\geq 0$ and all \nonperiodic $q'\leq q$, and hence
$\psi(v,u_0)_q$ is periodic too.

Thus, if we apply $\psi$, we have that in $(v,u_1):=\psi(v,u_0)$ all
points in $\Eper\cup E_0\cup E_1$ become periodic, together with
possibly some points in $E_{01}$.  The latter points get in any case
  $u_1$-value equal to $0$. This can be seen as follows. If any
  such point gets $u_1$-value equal to $1$, then all points below it
  get the same $u_{1}$-value. Yet, by definition, these points are
  above some frontier point in $E_1$ and frontier points in $E_1$ get
  $u_1$-value $0$ by the second statement of Lemma~\ref{lem:period}.

If $p_0\in E_0$ has become
periodic, we are done; we are also done if the rank of $p_0$
increases, because this is precisely what we want.
If $p_0$ has not become periodic and its rank has not increased, then
now all the \nonperiodic points below $p_0$ in $(v,u_1)$ have
$u_1$-value $0$ (by the previous
  remark) and have the same rank as $p_0$. Thus, they are the set
$E_0$ computed in $(v, u_1)$ (instead of in $(v,u_0)$) and we know by
the same considerations as above that it is sufficient to apply $\psi$
once more to make them periodic.
\end{proof}
\vskip 1cm
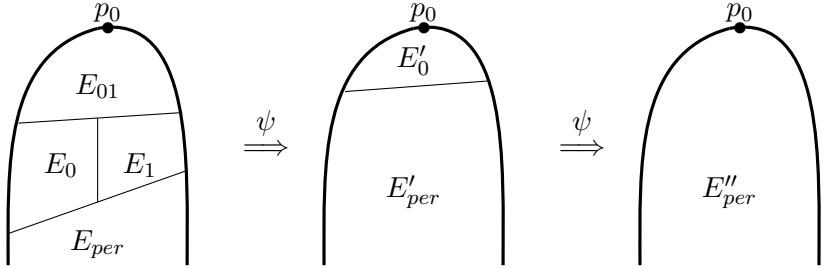
\begin{tikzpicture}[xscale=0.7,yscale=.7]
\draw[very thick] (0,0) to [out=90,in=195] (1.7,4.5); 
\draw[very thick] (1.7,4.5) to [out=10,in=90] (3.4,0); 
\draw (0,0.6) --(3.4,1.8); 
\draw (.2,2.7) --(3.3,2.9); 
\draw (1.7,1.2) --(1.7,2.8); 
\node at (1.9,4.8) {$p_0$};
\node at (1.9,4.5) {$\bullet$};
\node at (1.0,1.9) {$E_0$};
\node at (2.5,1.9) {$E_1$};
\node at (1.7,3.4) {$E_{01}$};
\node at (1.7,0.4) {$E_{per}$};

\node at (4.9,2.7) {$\psi$};
\node at (4.9,2.2) {$\Longrightarrow$};

\draw[very thick] (6,0) to [out=90,in=195] (7.7,4.5); 
\draw[very thick] (7.7,4.5) to [out=10,in=90] (9.4,0); 
\draw (6.4,3.3) --(9.1,3.5); 
\node at (7.9,4.8) {$p_0$};
\node at (7.9,4.5) {$\bullet$};
\node at (7.7,3.9) {$E'_0$};
\node at (7.7,1.4) {$E'_{per}$};

\node at (10.9,2.7) {$\psi$};
\node at (10.9,2.2) {$\Longrightarrow$};

\draw[very thick] (12,0) to [out=90,in=195] (13.7,4.5); 
\draw[very thick] (13.7,4.5) to [out=10,in=90] (15.4,0); 
\node at (13.9,4.8) {$p_0$};
\node at (13.9,4.5) {$\bullet$};
\node at (13.7,1.4) {$E''_{per}$};
\end{tikzpicture}

\vskip 1cm

Notice that some crucial arguments used in the above proof  (starting from  the induction on $\vert e\vert$ itself) make essential use of the fact that evaluations are order-preserving, so such arguments are not suitable for modal logics.

The above proof of Theorem~\ref{thm:main} gives a bound for $N$ which
is not optimal, when compared with the bound obtained via syntactic
means in~\cite{Ruitenburg84} (the syntactic computations
in~\cite{fossacs} for fixpoints convergence are also better). Thus
refining indexes of ultimate periodicity of our sequences within
semantic arguments remains as open question.

\newpage

\bibliographystyle{abbrv}
\bibliography{biblio,biblio2}

\end{document}